\documentclass[11pt,a4paper]{article}

\usepackage[T1]{fontenc}
\usepackage[utf8]{inputenc}
\usepackage[margin=2.7cm]{geometry}
\usepackage{amsmath,amssymb}
\usepackage[authoryear,longnamesfirst]{natbib}
\usepackage{algorithm}
\usepackage{algpseudocode}
\usepackage{subcaption}
\usepackage{tikz}
\usetikzlibrary{3d, positioning}
\usepackage[colorlinks=true,citecolor=blue,linkcolor=blue,urlcolor=blue]{hyperref}
\usepackage{doi}

\colorlet{poscolor}{white}
\colorlet{negcolor}{black!85}

\newcommand{\drawtensorlayer}[2]{
  \pgfmathsetmacro{\zpos}{(#1 - 1) * 3.4}

  \draw[black, thick] (0, 0, \zpos) -- (8, 0, \zpos) -- (8, 8, \zpos) -- (0, 8, \zpos) -- cycle;

  \foreach \val [count=\idx] in {#2} {
    \pgfmathsetmacro{\r}{int(floor((\idx-1)/8))}
    \pgfmathsetmacro{\c}{int(mod(\idx-1, 8))}
    \pgfmathsetmacro{\xmin}{\c}
    \pgfmathsetmacro{\xmax}{\c+1}
    \pgfmathsetmacro{\ymin}{7-\r}
    \pgfmathsetmacro{\ymax}{8-\r}

    \ifnum\val=1
      \fill[poscolor] (\xmin, \ymin, \zpos) -- (\xmax, \ymin, \zpos) -- (\xmax, \ymax, \zpos) -- (\xmin, \ymax, \zpos) -- cycle;
    \else
      \fill[negcolor] (\xmin, \ymin, \zpos) -- (\xmax, \ymin, \zpos) -- (\xmax, \ymax, \zpos) -- (\xmin, \ymax, \zpos) -- cycle;
    \fi
    \draw[black!40, ultra thin] (\xmin, \ymin, \zpos) -- (\xmax, \ymin, \zpos) -- (\xmax, \ymax, \zpos) -- (\xmin, \ymax, \zpos) -- cycle;
  }
}

\newtheorem{theorem}{Theorem}
\newtheorem{lemma}[theorem]{Lemma}
\newtheorem{proposition}[theorem]{Proposition}
\newtheorem{corollary}[theorem]{Corollary}
\newtheorem{example}[theorem]{Example}
\newtheorem{definition}[theorem]{Definition}

\newenvironment{proof}[1][Proof]{\noindent\textit{#1.}\ }{\par\vspace{6pt}}

\title{A computational framework for multidimensional loop-cocyclic Hadamard matrices}

\author{
Manuel Gonz\'alez-Regadera\thanks{Corresponding author. Department of Applied Mathematics I, Universidad de Sevilla, Spain. \texttt{mgonzalez34@us.es}}
\and
Ra\'ul M. Falc\'on\thanks{Department of Applied Mathematics I, Universidad de Sevilla, Spain. \texttt{rafalgan@us.es}}
}

\date{}

\begin{document}

\maketitle

\begin{abstract}
The cocyclic development of Hadamard matrices has recently been extended from groups to loops by means of a cohomology theory that incorporates associativity obstructions. Over the finite field $\mathbb F_2$, the resulting loop-cocycles can be computed as solutions of homogeneous linear systems, making the framework suitable for exact computation. In this paper, we investigate multidimensional Hadamard matrices arising from finite loops. We introduce the notion of $\delta$-compatible pairs of $2$-cochains over a loop $L$ as a natural extension of the usual $2$-cocycle identity, which allows the use of two distinct $2$-cochains. They form a vector space naturally isomorphic to the direct product of the space of $2$-cocycles and the space of $1$-cochains. This yields a generalization of the classical cocyclic construction of multidimensional Hadamard matrices from groups to arbitrary loops, while preserving the computational advantages of the cocyclic approach. The obtained decomposition shows that the determination of Hadamard equivalence classes can be reduced from the full space of $\delta$-compatible pairs to the smaller space of ordinary $2$-cocycles, eliminating a redundant factor of $2^n$ from the search space. Based on this result, we develop and implement an algorithm in {\sc GAP}/{\sc GRAPE} for the construction and equivalence classification of multidimensional Hadamard matrices arising from finite loops. Computational results for all groups of orders $4$ and $8$, together with a non-associative loop of order $8$, show that the proposed three-dimensional construction refines the classical cocyclic Hadamard classification. While all examples collapse into a single equivalence class in dimension two, they split into several distinct classes in dimension three. In particular, the non-associative loop produces a three-dimensional Hadamard class that does not arise from any of the groups considered, showing that the multidimensional construction detects structural information that is invisible at the cocyclic matrix level.
\end{abstract}


\section{Introduction}\label{sec:Introduction}

A {\em Hadamard matrix} $H$ of order $n$ is an $n\times n$ array, with entries in the cyclic group $\langle\,-1\,\rangle = \{-1,1\}$, such that $HH^T=nI_n$. Here, $I_n$ denotes the identity matrix of order $n$. In particular, $n$ must be $1$, $2$ or a multiple of $4$. A Hadamard matrix is said to be {\em normalized} if all elements in its first row and first column are $1$. Every Hadamard matrix is {\em Hadamard equivalent} to a normalized one. That is, one can be obtained from the other by permutations and negations of rows and columns. The {\em Hadamard conjecture} states that there is a Hadamard matrix of order $4t$ for every positive integer $t$. A brute-force search is computationally intractable because it requires evaluating $2^{n^{2}}$ possible configurations.

\newpage

In the literature, there are different algebraic approaches for constructing infinite families of Hadamard matrices (see in this regard the handbook of \citealp{Horadam2007}). Since not all orders are reached by these methods, the difficulty spikes for those ones not covered by these algebraic approaches. In the last $20$ years, the smallest unsolved case has been $n = 668$. In August 2026, however, \citealp{alpoge2026} have successfully generated computationally verified examples of all $12$ previously unknown Hadamard matrices of sizes under 2000 (see also \citealp{epoch2026}).

\citealp{Shlichta71, Shlichta79} generalized the notion of a Hadamard matrix to any higher dimension $d>2$, and pointed out that such arrays might have applications to encryption and error-correcting codes. An {\em improper $d$-dimensional Hadamard matrix} of order $n$ with entries in $\langle\,-1\,\rangle$ is any array
\[H:=\left(H\left[i_1,\ldots,i_d\right]\right)_{i_1,\ldots,i_d\in [n]}\]
with $[n]:=\{1,\ldots,n\}$, such that all parallel $(d-1)$-dimensional sections are mutually orthogonal. That is,
\[\sum_{i_1,\ldots,i_{k-1},i_{k+1},\ldots,i_d\in [n]}H\left[i_1,\ldots,i_{k-1},a,i_{k+1},\ldots,i_d\right]\cdot H\left[i_1,\ldots,i_{k-1},b,i_{k+1},\ldots,i_d\right]=n^{d-1}\cdot \delta_{ab}\]
for all $a,b,k\in [n]$, where $\delta_{ab}$ is the Kronecker delta related to $a,b$. This array is termed {\em proper} whenever all parallel lines in each planar section, in all axis-normal orientations, are mutually orthogonal. That is, for each pair of elements $k,l\in [n]$,
{\footnotesize \[\sum_{j\in [n]}H\left[i_1,\ldots,i_{k-1},a,i_{k+1},\ldots,i_{l-1},j,i_{l+1},\ldots,i_d\right]\cdot H\left[i_1,\ldots,i_{k-1},b,i_{k+1},\ldots,i_{l-1},j,i_{l+1},\ldots,i_d\right]=n\cdot \delta_{ab}\]}
for all $a,b,i_1,\ldots,i_{k-1},i_{k+1},\ldots,i_{l-1},i_{l+1},\ldots,i_d\in [n]$. Thus, this array exists only if there is a planar Hadamard matrix of the same order. Concerning improper Hadamard matrices, their order must be even. It does not need to be a multiple of four, but it is not currently known whether there is an improper $d$-dimensional Hadamard matrix of every even order.

A simple method to construct proper $d$-dimensional Hadamard matrices is the {\em component product}, which was independently described by \citealp{Yang86} and \citealp{deLauney87}. If $H$ is a Hadamard matrix, then the $d$-dimensional array $A$ described as
\[A[i_1,\ldots,i_d]:=\prod_{1\leq r<s\leq d} H[i_r,i_s]\]
is a proper $d$-dimensional Hadamard matrix. For $d=3$, \citealp{Yang01} proved the existence of improper Hadamard matrices of order $2\cdot 3^t$, with $t\geq 1$.

In 1995, \citealp{Horadam1995} introduced the {\em cocyclic development} of Hadamard matrices over a finite group $G$ as an efficient way to construct Hadamard matrices. Let us briefly recall the fundamentals of the additive framework of this theory, which makes use of the additive finite field $\mathbb{Z}_2:=(\mathbb{F}_2,+)$. The {\em coboundary operator} $\delta$ maps each function $f:G^k\to \mathbb{Z}_2$, with $k$ a positive integer, to a function $\delta_f:G^{k+1}\to \mathbb{Z}_2$ so that
\begin{equation}\label{eq:coboundary}
    \delta_f(a_1,\ldots,a_{k+1}):=f(a_1,\ldots,a_k)+f(a_2,\ldots,a_{k+1})+ \sum_{i=1}^kf(a_1,\ldots,a_ia_{i+1},\ldots,a_{k+1})
\end{equation}
for all $a_1,\ldots,a_{k+1}\in G$. The function $f$ is termed a {\em $k$-cochain} over $G$, while the function $\delta_f$ is termed a {\em $(k+1)$-coboundary} over $G$. If $\delta_f={\bf 0}$, then $f$ is termed a {\em $k$-cocycle} over $G$. The $k$-cochain $f$ is said to be {\em normalized} if $f(a_1,\ldots,a_k)=0$ whenever $a_i$ is the identity element of $G$ for some $i$. When no confusion arises, we simply refer to these functions as cochains, coboundaries, or cocycles. From now on, we denote the sets of $k$-cochains, $k$-coboundaries and $k$-cocycles over $G$ respectively by $\mathcal{C}^k(G)$, $\mathcal{B}^k(G)$ and $\mathcal{Z}^k(G)$. All of them are abelian groups under the pointwise product. Moreover, $\mathcal{B}^k(G)\trianglelefteq\mathcal{Z}^k(G)$. The quotient group

\[\mathcal{H}^k(G):=\mathcal{Z}^k(G)/\mathcal{B}^k(G)\]
is the {\em $k^{\mathrm{th}}$ cohomology group of $G$ over $\langle\,-1\,\rangle$}. This cohomology is useful for constructing $d$-dimensional Hadamard matrices based on $G$. To see it, for each $k$-cochain $f\in \mathcal{C}^k(G)$, we define the matrix $M_f$ so that
\begin{equation}\label{eq:cocyclicHM}
    M_f[a_1,\ldots,a_k]:=(-1)^{f(a_1,\ldots,a_k)}
\end{equation}
for all $a_1,\ldots,a_k\in G$. A $d$-dimensional Hadamard matrix $H$, with $d\geq 2$, is termed {\em $d$-cocyclic over $G$} if there is a $d$-cocycle $\psi\in \mathcal{Z}^d(G)$ such that $H$ coincides with the {\em cocyclic matrix} $M_\psi$. For $d=2$, \citealp{Horadam1995} proved the so-called {\em cocyclic Hadamard test}, which ensures that if $M_\psi$ is normalized, then it is Hadamard if and only if the sum of all entries in each row is zero, except for the constant row. Thus, determining whether a cocyclic normalized matrix is Hadamard is computationally much faster than the general case.

\citealp{OC11} computationally enumerated all $2$-cocyclic Hadamard matrices of order less than 40. In particular, all Hadamard matrices of order up to 20 are Hadamard equivalent to a $2$-cocyclic normalized one. Moreover, 16 of the 60 Hadamard equivalence classes of order 24, and six of the 487 classes of order 28 are $2$-cocyclic. In addition, there are 100 $2$-cocyclic Hadamard equivalence classes of order 32, and 35 of order 36. Furthermore, \citealp{Horadam1998} described a simple method to construct proper $d$-dimensional Hadamard matrices from $2$-cocycles over finite groups. If the cocyclic matrix $M_\psi$ is Hadamard, then there is a proper $d$-dimensional Hadamard matrix $H$ such that
\begin{equation}\label{eq:2tod}
H[a_1,\ldots,a_d]:=(-1)^{\sum_{k=2}^d \psi\left(\prod_{j=1}^{k-1}a_j,\,a_k\right)}
\end{equation}
for all $a_1,\ldots,a_d\in G$. \citealp{Alvarez2015} proved that this array is not $d$-cocyclic in general. For $d=3$, they computationally proved the existence of $64$ improper three-dimensional $3$-cocyclic Hadamard matrices on $\mathbb{Z}_2^2$, and $32$ improper ones on $\mathbb{Z}_4$. None of them is proper. To the best of our knowledge, no other article has addressed this issue, even if one may find some recent papers concerning high dimensional Hadamard matrices (see \citealp{Krcadinac2023,Krcadinac2025}). Moreover, a natural extension of the classical cocyclic Hadamard test for $2$-dimensional $2$-cocyclic Hadamard matrices is currently an open problem for dimension $d>2$.

The cocyclic development of Hadamard matrices over a finite group has been naturally generalized to non-associative loops by \citealp{Alvarez2019}, \citealp{Alvarez2020} and \citealp{Falcon2021}. Recall that a {\em loop} of order $n$ is a finite set $L$ of $n$ elements that is endowed with a binary operation with an identity element such that both equations $ax=b$ and $ya=b$ have unique solutions $x,y\in L$ for all $a,b\in L$. Such a loop is a group if and only if the binary operation is associative. Its multiplication table is a {\em Latin square} of order $n$, that is, an $n\times n$ array such that each symbol of $L$ occurs exactly once in each row and exactly once in each column. The respective sets $\mathcal{C}^k(L)$, $\mathcal{B}^k(L)$ and $\mathcal{Z}^k(L)$ of $k$-cochains, $k$-coboundaries and $k$-cocycles over $L$, as well as the notion of $d$-cocyclic Hadamard matrix over $L$, are defined analogously to those in the associative case once the underlying group $G$ is replaced by the loop $L$. The three mentioned sets are abelian groups under the pointwise product. A relevant aspect here is that the cocyclic Hadamard test is still valid in this non-associative context (see \citealp{Alvarez2019}). Despite this, the classical cohomology is not valid for non-associative loops because, unlike the associative case, $\mathcal{B}^k(L)\not \subseteq \mathcal{Z}^k(L)$ (see \citealp{Alvarez2020}).

\newpage

There are two different approaches concerning cohomology of non-associative loops: \citealp{Eilenberg1947} focused on the {\em associator} $\mathcal{A}(a,b,c):=((ab)c)/(a(bc))$, while \citealp{Johnson1990} focused on module extensions. Based on the first approach, a formal framework for the cocyclic development of Hadamard matrices over loops has recently been introduced by \citealp{Falcon26}. They have introduced the notion of a {\em $k$-loop-cocycle} over a finite loop $L$, with $k\geq 2$, as a $k$-cochain $\psi\in\mathcal{C}^k(L)$ such that $\delta_\psi=\delta^2(f)$ for some $f\in\mathcal{C}^{k-1}(L)$. This is a $k$-cocycle whenever $L$ is a group, because $\delta^2(f)={\bf 0}$ in that case. The coboundary operator therefore measures the lack of associativity of a loop as an obstruction that prevents the $k$-cochain from being a cocycle. From now on, we denote by $\mathcal{Z}_{\mathcal{L}}^k(L)$ the abelian group under the pointwise product of $k$-loop-cocycles of $L$ over $\mathbb{Z}_2$ and let $\mathcal C_0^k(L)$ denote the subgroup of normalized $k$-cochains. \citealp{Falcon26} proved that $\mathcal{B}^k(L)\trianglelefteq \mathcal{Z}_{\mathcal{L}}^k(L)$ and $\mathcal{Z}^k(L)\trianglelefteq \mathcal{Z}_{\mathcal{L}}^k(L)$. They also defined the abelian subgroup
    \[\mathcal{B}_0^k(L):=\left\{\delta_f\in \mathcal{B}^k(L)\colon\, f\in \mathcal{C}^{k-1}_0(L)\right\}\trianglelefteq \mathcal{B}^k(L)\]
    and the quotient group
    \[\mathcal{B}_\mathcal{L}^k(L):=\mathcal{B}_0^k(L)/\left(\mathcal{B}_0^k(L)\cap \mathcal{Z}^k(L)\right).\]
Then, they proved that
\begin{equation}\label{eq:exact}
\mathcal{Z}^k(L)\cong\mathcal{Z}_{\mathcal{L}}^k(L)/\mathcal{B}_{\mathcal{L}}^k(L).
\end{equation}
The abelian groups $\mathcal{Z}^k(L)$, $\mathcal{Z}_{\mathcal{L}}^k(L)$ and $\mathcal{B}_{\mathcal{L}}^k(L)$ are vector spaces over the Galois field $\mathbb{F}_2$, once each loop-cocycle $\psi\in\mathcal{Z}_{\mathcal{L}}^k(L)$ is represented as a solution of the homogeneous linear system of equations described by the condition $\delta_\psi\in \mathrm{Im}(\delta^2)$ (see \citealp{Falcon26}). As a consequence, we have from \eqref{eq:exact} that
\begin{equation}\label{eq:decomposition}
\mathcal{Z}_\mathcal{L}^k(L)\cong \mathcal{Z}^k(L)\oplus \mathcal{B}_\mathcal{L}^k(L)
\end{equation}
and hence,
\[\dim_{\mathbb{F}_2}\left(\mathcal{Z}_\mathcal{L}^k(L)\right)=\dim_{\mathbb{F}_2}\left(\mathcal{Z}^k(L)\right)+\dim_{\mathbb{F}_2}\left(\mathcal{B}_\mathcal{L}^k(L)\right).\]
Moreover, if $\psi\in \mathcal{Z}_\mathcal{L}^k(L)$ satisfies $\delta_\psi=\delta^2(f)$ for some $f\in \mathcal{C}^{k-1}(L)$, then
\begin{equation}\label{eq:psi_delta}
\psi=(\psi+\delta_f)+\delta_f
\end{equation}
where $\psi+\delta_f\in \mathcal{Z}^k(L)$ and $\delta_f\in \mathcal{B}_0^k(L)$.

A $d$-dimensional Hadamard matrix $H$, with $d\geq 2$, is $d$-loop-cocyclic over $L$ if there is a $d$-loop-cocycle $\psi\in \mathcal{Z}_\mathcal{L}^d(L)$ such that $H$ coincides with the loop-cocyclic matrix $M_\psi$ defined as in \eqref{eq:cocyclicHM}. The classical notion of a cocyclic Hadamard matrix arises naturally when $L$ is a group. The case $d=2$ has been comprehensively studied by \citealp{Falcon26}, who computationally proved the existence of four new Hadamard equivalence classes of order 24 and eight new classes of order 28, which are not cocyclic over any finite group, but they are loop-cocyclic over non-associative loops. They also proved that the classical cocyclic Hadamard conjecture does not hold as a characterization. More precisely, the associated matrix $M_\psi$ to a $2$-loop-cocycle $\psi\in \mathcal{Z}_{\mathcal{L}}^2(L)$ is a Hadamard matrix only if
\[\sum_{b\in  L} M_\psi[a,b]=0\]
whenever $a\in L\setminus\{e\}$, but the converse is valid only if $\psi\in \mathcal{Z}^2(L)$.

\newpage

This paper investigates multidimensional loop-cocyclic Hadamard matrices in dimension $d>2$. Our main objective is to extend the classical cocyclic construction from groups to arbitrary finite loops while preserving its computational tractability. To this end, in Section~\ref{sec:compatible} we introduce the notion of $\delta$-compatible pairs of $2$-cochains, which generalizes the usual $2$-cocycle identity and allows the simultaneous use of two distinct $2$-cochains. We also establish their algebraic structure and show that they admit a decomposition in terms of ordinary $2$-cocycles and $1$-cochains. In Section~\ref{sec:2_loop_cocycles}, we use this framework to generalize the classical multidimensional cocyclic construction from groups to loops, obtaining proper multidimensional Hadamard matrices from $\delta$-compatible pairs of $2$-cochains. We then relate this construction to the theory of $2$-loop-cocycles and show that the resulting multidimensional Hadamard equivalence classes can be studied through ordinary cocycles without loss of generality. Finally, Section~\ref{sec:computational} presents a computational framework implemented in {\sc GAP}/{\sc GRAPE} for the construction and classification of the resulting Hadamard arrays, together with computational results for groups and non-associative loops of small order.

\section{$\delta$-compatible $2$-cochains over loops}\label{sec:compatible}

In this section, we introduce the notion of a pair of $2$-cochains over a finite loop $L$ of order $n$, with an identity element $e$, which are compatible with respect to the coboundary operator $\delta$.

\begin{definition}\label{definition:compatible} A pair $(\psi,\phi)\in \mathcal{C}^2(L)\times \mathcal{C}^2(L)$ is said to be {\em $\delta$-compatible over $L$} whenever
\begin{equation}\label{eq:psi_phi}
\psi(a,b)+\phi(ab,c)=\phi(a,bc)+\psi(b,c)
\end{equation}
holds for all $a,b,c\in L$. It is said to be {\em symmetric} if the pair $(\phi,\psi)$ is also $\delta$-compatible. The set of all $\delta$-compatible pairs of $2$-cochains over $L$ is denoted by $\mathcal{C}^2_\delta(L)$, while its subset of symmetric pairs is denoted by $\mathcal{S}\mathcal{C}^2_\delta(L)$.
\end{definition}

This compatibility condition extends the usual $2$-cocycle identity described by \eqref{eq:coboundary} by allowing two distinct $2$-cochains. More precisely, if $\psi=\phi$ in Definition \ref{definition:compatible}, then $\psi\in\mathcal{Z}^2(L)$. That is,
\[\mathcal{Z}^2(L) \subseteq \mathcal{S}\mathcal{C}^2_\delta(L) \subseteq \mathcal{C}^2_\delta(L).\]
The first inclusion is, indeed, strictly proper. To see this, consider the constant $2$-cochains $\mathbf{0}, \mathbf{1} \in \mathcal{C}^2(L)$ defined by $\mathbf{0}(a,b) = 0$ and $\mathbf{1}(a,b) = 1$ for all $a,b \in L$. It is easily verified that $(\mathbf{0}, \mathbf{1}) \in \mathcal{S}\mathcal{C}^2_\delta(L)\setminus \mathcal{Z}^2(L)$. The next example illustrates a pair of $2$-cochains in $\mathcal{C}^2_\delta(L)\setminus \mathcal{S}\mathcal{C}^2_\delta(L)$.

\begin{example}\label{example_0} Let $L$ be the non-associative loop having the following Latin square as its multiplication table.
\[L\equiv\begin{array}{|c|c|c|c|c|c|c|c|}\hline
1 & 2 & 3 & 4 & 5 & 6 & 7 & 8\\ \hline
2 & 1 & 4 & 3 & 6 & 5 & 8 & 7\\ \hline
3 & 4 & 1 & 2 & 7 & 8 & 5 & 6\\ \hline
4 & 3 & 2 & 1 & 8 & 7 & 6 & 5\\ \hline
5 & 6 & 8 & 7 & 3 & 4 & 2 & 1\\ \hline
6 & 5 & 7 & 8 & 4 & 3 & 1 & 2\\ \hline
7 & 8 & 6 & 5 & 1 & 2 & 4 & 3\\ \hline
8 & 7 & 5 & 6 & 2 & 1 & 3 & 4\\ \hline
\end{array}.\]
It was the first non-associative loop over which a cocycle was defined (see \citealp{Alvarez2019}). In addition, we consider the $2$-cochains $\psi,\phi\in\mathcal{C}^2(L)$ whose associated matrices, as defined in \eqref{eq:cocyclicHM}, are

{\scriptsize \[M_\psi=\left(\begin{array}{cccccccc}
+ & + & + & + & + & + & + & +\\
+ & + & + & + & + & + & + & +\\
+ & + & - & - & + & + & - & -\\
+ & + & - & - & + & + & - & -\\
+ & + & + & + & + & + & - & -\\
+ & + & + & + & + & + & - & -\\
+ & + & - & - & - & - & - & -\\
+ & + & - & - & - & - & - & -\\
\end{array}\right)\hspace{1cm}\text{and}\hspace{1cm} M_\phi=\left(\begin{array}{cccccccc}
+ & - & - & + & - & - & + & +\\
- & + & + & - & - & - & + & +\\
- & + & - & + & + & + & + & +\\
+ & - & + & - & + & + & + & +\\
- & - & + & + & - & + & + & -\\
- & - & + & + & + & - & - & +\\
+ & + & + & + & - & + & - & +\\
+ & + & + & + & + & - & + & -\\
\end{array}\right).\]}

In both matrices, the entries $1$ and $-1$ are respectively represented by the signs $+$ and $-$. Notice that $\psi$ is normalized, but $\phi$ is not. It can be verified that the pair $(\psi,\phi)$ is $\delta$-compatible. Thus, for instance,

{\footnotesize\[\psi(2,4)+\phi(2\cdot 4,7)=\psi(2,4)+\phi(3,7)=0+0\equiv 0\pmod 2\equiv 1+1=\phi(2,6)+\psi(4,7)=\phi(2,4\cdot 7)+\psi(4,7).\]}

Nevertheless, the pair $(\psi,\phi)$ is not symmetric because

{\footnotesize\[\phi(1,1)+\psi(1\cdot 1,2)=\phi(1,1)+\psi(1,2)=0+0\equiv 0 \pmod 2\not\equiv 0+1=\psi(1,2)+\phi(1,2)=\psi(1,1\cdot 2)+\phi(1,2).\]}
\end{example}

\vspace{0.2cm}

In what follows, we show some preliminary results concerning the notion of $\delta$-compatibility.

\begin{lemma}\label{lemma:abelian_group_delta}
Both sets $\mathcal{C}^2_\delta(L)$ and $\mathcal{S}\mathcal{C}^2_\delta(L)$ are abelian groups under the pointwise product of $2$-cochains.
\end{lemma}

\begin{proof}
Let $(\psi_1, \phi_1), (\psi_2, \phi_2) \in \mathcal{C}^2_\delta(L)$. Then,
\begin{align*}
(\psi_1 + \psi_2)(a,b) + (\phi_1 + \phi_2)(ab,c) &= \big(\psi_1(a,b) + \phi_1(ab,c)\big) + \big(\psi_2(a,b) + \phi_2(ab,c)\big) \\
&= \big(\phi_1(a,bc) + \psi_1(b,c)\big) + \big(\phi_2(a,bc) + \psi_2(b,c)\big) \\
&= (\phi_1 + \phi_2)(a,bc) + (\psi_1 + \psi_2)(b,c),
\end{align*}
for all $a,b,c \in L$. Thus, $(\psi_1 + \psi_2, \phi_1 + \phi_2) \in \mathcal{C}^2_\delta(L)$. Associativity and commutativity follow directly from the abelian group structure of $\mathbb{Z}_2$. The identity element is the zero pair $({\bf 0},{\bf 0})\in\mathcal{C}^2(L)\times \mathcal{C}^2(L)$, and every element is its own inverse since $x = -x$ in $\mathbb{Z}_2$. Consequently, $\mathcal{C}^2_\delta(L)$ is an elementary abelian $2$-group. The proof for $\mathcal{S}\mathcal{C}^2_\delta(L)$ follows identically by applying the same linearity argument to the swapped condition. \hfill $\Box$
\end{proof}

\begin{lemma}\label{lem:fixed_phi} If $(\psi_1,\phi),\,(\psi_2,\phi)\in\mathcal{C}_\delta^2(L)$, then $\psi_1+\psi_2=\lambda\cdot{\bf 1}$, with $\lambda\in \mathbb{Z}_2$.
\end{lemma}

\begin{proof} Since $(\psi_1,\phi),\,(\psi_2,\phi)\in\mathcal{C}_\delta^2(L)$, we have from \eqref{eq:psi_phi}
\[\psi_1(a,b)+\psi_1(b,c)=\phi(ab,c)+\phi(a,bc)=\psi_2(a,b)+\psi_2(b,c)\]
for all $a,b,c\in L$. Then, $(\psi_1+\psi_2)(a,b)=(\psi_1+\psi_2)(b,c)$ for all $a,b,c\in L$. As a consequence, $(\psi_1+\psi_2)(a,b)$ does not depend on the element $a$, and $(\psi_1+\psi_2)(b,c)$ does not depend on the element $c$. Thus, there must exist a pair of maps $g,h\in\mathcal{C}^1(L)$ such that $(\psi_1+\psi_2)(a,b)=g(b)$ and $(\psi_1+\psi_2)(b,c)=h(b)$ for all $a,b,c\in L$. But $(\psi_1+\psi_2)(a,b)=(\psi_1+\psi_2)(b,c)$ for all $a,b,c\in L$, so $g=h$. Moreover, taking $a:=b$ in the identity $(\psi_1+\psi_2)(a,b)=(\psi_1+\psi_2)(b,c)$ yields $(\psi_1+\psi_2)(b,b)=(\psi_1+\psi_2)(b,c)$ for all $b,c\in L$, that is, $g(b)=g(c)$ for all $b,c\in L$. Hence, $g$ is constant. \hfill $\Box$
\end{proof}

\begin{lemma}\label{lem:equivalence} It is verified that $(\psi,\phi)\in\mathcal{C}_\delta^2(L)$ if and only if
\begin{equation}\label{eq:equivalence}
\delta_\phi(a,b,c)=(\psi+\phi)(a,b)+(\psi+\phi)(b,c)
\end{equation}
for all $a,b,c\in L$.
\end{lemma}

\begin{proof} We claim that \eqref{eq:equivalence} is equivalent to \eqref{eq:psi_phi}. To prove it, let $a,b,c\in L$. Then,
\begin{align*}
\delta_\phi(a,b,c)=(\psi+\phi)(a,b)+(\psi+\phi)(b,c) & \Leftrightarrow\\
\phi(a,b)+\phi(ab,c)+\phi(a,bc)+\phi(b,c)=\psi(a,b)+\phi(a,b)+\psi(b,c)+\phi(b,c) & \Leftrightarrow\\
\phi(ab,c)+\phi(a,bc)=\psi(a,b)+\psi(b,c) & \Leftrightarrow\\
\psi(a,b)+\phi(ab,c)=\phi(a,bc)+\psi(b,c). &
\end{align*}
\hfill $\Box$
\end{proof}

\begin{lemma}\label{lemma:compatible} If $(\psi,\phi)\in\mathcal{C}_\delta^2(L)$, then the following statements hold.
\begin{enumerate}
\item $\psi(a,e)=\psi(e,b)$ for all $a,b\in L$.

\item If both $\psi$ and $\phi$ are normalized, then $\psi=\phi$.

\item The pair $(\psi,\phi)$ is symmetric if and only if $\psi+\phi = \lambda \cdot \bf 1$, with $\lambda\in\mathbb{Z}_2$. If this is the case, then $\psi,\phi\in\mathcal{Z}^2(L)$.
\end{enumerate}
\end{lemma}

\begin{proof} We prove each statement separately.

\begin{enumerate}
\item It follows easily from taking $b=e$ in \eqref{eq:psi_phi}.

\item If we take $c=e$ in \eqref{eq:psi_phi}, then
\[\psi(a,b)=\psi(a,b)+\phi(ab,e)=\phi(a,b)+\psi(b,e)=\phi(a,b)\]
for all $a,b\in L$. Thus, the second statement holds.

\item First, we assume that $(\psi,\phi)\in\mathcal{SC}_\delta^2(L)$. From the first statement, there exist two elements $r,s\in\mathbb{Z}_2$ such that $\psi(a,e)=r$ and $\phi(a,e)=s$ for all $a\in L$. If we take $c=e$ in \eqref{eq:psi_phi}, then $\psi(a,b)+s=\phi(a,b)+r$ for all $a,b\in L$. Thus, the necessary condition follows from defining $\lambda:=r+s$. If this is the case, then
\[\psi(a,b)+\psi(ab,c)= \psi(a,b)+\phi(ab,c)+\lambda\cdot{\bf 1} =\phi(a,bc)+\psi(b,c)+\lambda\cdot{\bf 1}=\psi(a,bc)+\psi(b,c)\]
for all $a,b,c\in L$. Hence, $\psi\in\mathcal{Z}^2(L)$. Similarly, $\phi\in\mathcal{Z}^2(L)$.

Now, we assume that $\psi+\phi = \lambda \cdot \bf 1$, with $\lambda\in\mathbb{Z}_2$. In particular, we have just proved that $\psi\in\mathcal{Z}^2(L)$. Then,
\[\phi(a,b)+\psi(ab,c)=(\psi(a,b)+\lambda) +\psi(ab,c)=\psi(ab,c)+(\psi(b,c)+\lambda)=\psi(ab,c)+\phi(b,c)\]
for all $a,b,c\in L$. Thus, $(\psi,\phi)\in\mathcal{SC}_\delta^2(L)$. \hfill $\Box$
\end{enumerate}
\end{proof}

\begin{proposition}\label{proposition:compatible} Let $(\psi,\phi)\in\mathcal{C}_\delta^2(L)$ be such that $\phi\in\mathcal{Z}^2(L)$. Then, $(\psi,\phi)\in\mathcal{SC}^2_\delta(L)$.
\end{proposition}

\begin{proof} Since $(\psi,\phi)\in\mathcal{C}_\delta^2(L)$ and $\phi\in\mathcal{Z}^2(L)$, we have from \eqref{eq:psi_phi}
\[0=\psi(a,b)+\psi(b,c)+\phi(ab,c)+\phi(a,bc)=\psi(a,b)+\psi(b,c)+\phi(a,b)+\phi(b,c)\]
for all $a,b,c\in L$. That is,
\[(\psi+\phi)(a,b)=(\psi+\phi)(b,c)\]
for all $a,b,c\in L$. On the one hand, if $a=b$, then
\[(\psi+\phi)(b,b)=(\psi+\phi)(b,c)\]
for all $b,c\in L$. On the other hand, if $c=b$, then
\[(\psi+\phi)(a,b)=(\psi+\phi)(b,b)\]
for all $a,b\in L$. Thus, $\psi+\phi=\lambda\cdot{\bf 1}$, and hence the result follows from the third statement of Lemma \ref{lemma:compatible}. \hfill $\Box$
\end{proof}

\begin{lemma}\label{lemma:compatible_a} Let $\psi\in \mathcal{Z}_\mathcal{L}^2(L)$ and $f\in \mathcal{C}^1(L)$ be such that $\delta_\psi=\delta^2(f)$. In addition, let $\phi_{\psi,f}\in\mathcal{C}^2(L)$ be defined so that
    \begin{equation}\label{eq:psi_f}
    \phi_{\psi,f}(a,b):=\psi(a,b)+f(ab)
    \end{equation}
    for all $a,b\in L$. Then, $(\psi,\phi_{\psi,f})\in\mathcal{C}^2_\delta(L)$.
\end{lemma}

\begin{proof} Let $a,b,c\in L$. Then,
\begin{align*}
 \psi(a,b)+\phi_{\psi,f}(ab,c) & = \psi(a,b)+\psi(ab,c)+f((ab)c)= \tag*{(From \eqref{eq:psi_f})}\\
 & =\psi(a,bc)+\psi(b,c)+f(a(bc))= \tag*{($\psi\in \mathcal{Z}_\mathcal{L}^2(L)$ and $\delta_\psi=\delta^2(f)$)}\\
 &=\phi_{\psi,f}(a,bc)+\psi(b,c). \tag*{(From \eqref{eq:psi_f})}
\end{align*}
\hfill $\Box$
\end{proof}

The following two results show how the abelian groups $\mathcal{C}^2_\delta(L)$ and $\mathcal{S}\mathcal{C}^2_\delta(L)$ are indeed vector subspaces of $\mathcal{C}^2(L) \times \mathcal{C}^2(L)$ over the Galois field $\mathbb{F}_2$ once each pair of $2$-cochains is represented by a solution of the homogeneous linear system of equations described by the condition \eqref{eq:psi_phi}.

\begin{lemma}\label{lemma_delta_compatible_system}
The following statements hold.
\begin{enumerate}
    \item The set $\mathcal{C}^2_\delta(L)$ is identified with the subspace of solutions in $\mathbb{Z}_2^{2n^2}$ of the homogeneous linear system
    \begin{equation}\label{eq:system_delta}
    x_{a,b} + x_{b,c} + y_{ab,c} + y_{a,bc} = 0, \quad \text{for all } a,b,c \in L,
    \end{equation}
    over $\mathbb{Z}_2$ in the set of variables $\left\{x_{a,b}\colon\, a,b\in L\right\}\cup \left\{y_{a,b}\colon\, a,b\in L\right\}$.

    \item The set $\mathcal{S}\mathcal{C}^2_\delta(L)$ is identified with the subspace of solutions in $\mathbb{Z}_2^{2n^2}$ of the homogeneous linear system
    \begin{equation}\label{eq:system_delta_sym}
    \begin{cases}
    x_{a,b} + x_{b,c} + y_{ab,c} + y_{a,bc} = 0, \\
    y_{a,b} + y_{b,c} + x_{ab,c} + x_{a,bc} = 0,
    \end{cases}
    \quad \text{for all } a,b,c \in L,
    \end{equation}
    over $\mathbb{Z}_2$ in the same set of variables.
\end{enumerate}
\end{lemma}

\begin{proof}
Every solution of \eqref{eq:system_delta} or \eqref{eq:system_delta_sym} is a vector
\[
s := \left(s_{0,0}^{(x)}, \ldots, s_{n-1,n-1}^{(x)}, s_{0,0}^{(y)}, \ldots, s_{n-1,n-1}^{(y)}\right) \in \mathbb{Z}_2^{2n^2},
\]
where $s_{a,b}^{(x)}$ corresponds to $x_{a,b}$ and $s_{a,b}^{(y)}$ corresponds to $y_{a,b}$. The vector $s$ uniquely identifies a pair $(\psi_s, \phi_s) \in \mathcal{C}^2(L) \times \mathcal{C}^2(L)$, defined by $\psi_s(a,b) := s_{a,b}^{(x)}$ and $\phi_s(a,b) := s_{a,b}^{(y)}$ for all $a,b \in L$. Evaluating \eqref{eq:system_delta} over $\mathbb{Z}_2$ shows that $\psi_s(a,b) + \psi_s(b,c) + \phi_s(ab,c) + \phi_s(a,bc) = 0$ for all $a,b,c \in L$, which is equivalent to \eqref{eq:psi_phi}, which proves the first statement. Concerning the second one, adding the second equation $y_{a,b} + y_{b,c} + x_{ab,c} + x_{a,bc} = 0$ guarantees that $\phi_s(a,b) + \phi_s(b,c) + \psi_s(ab,c) + \psi_s(a,bc) = 0$ also holds for all $a,b,c \in L$, which asserts that $(\phi_s, \psi_s) \in \mathcal{C}^2_\delta(L)$ and hence $(\psi_s, \phi_s) \in \mathcal{S}\mathcal{C}^2_\delta(L)$. \hfill $\Box$
\end{proof}

\begin{proposition}\label{prop:vector_space_delta} The sets
$\mathcal{C}^2_\delta(L)$ and $\mathcal{S}\mathcal{C}^2_\delta(L)$ are vector spaces over $\mathbb{F}_2$.
\end{proposition}

\begin{proof}
By definition, $\mathcal{C}^2(L) \times \mathcal{C}^2(L)$ is isomorphic to the $2n^2$-dimensional vector space $\mathbb{Z}_2^{2n^2}$ over $\mathbb{F}_2$. By Lemma~\ref{lemma_delta_compatible_system}, $\mathcal{C}^2_\delta(L)$ and $\mathcal{S}\mathcal{C}^2_\delta(L)$ are respectively identified with the solution spaces of the homogeneous linear systems \eqref{eq:system_delta} and \eqref{eq:system_delta_sym} over $\mathbb{Z}_2$. Since the kernel of any linear operator over $\mathbb{F}_2$ is a subspace, both sets $\mathcal{C}^2_\delta(L)$ and $\mathcal{S}\mathcal{C}^2_\delta(L)$ are linear subspaces of $\mathcal{C}^2(L) \times \mathcal{C}^2(L)$, and the result holds. \hfill $\Box$
\end{proof}

In what follows, we characterize both vector spaces $\mathcal{C}^2_\delta(L)$ and $\mathcal{SC}^2_\delta(L)$.

\begin{theorem}\label{theorem:compatible_isomorphism} There is a vector space isomorphism
\[\mathcal{C}^2_\delta(L)\cong \mathcal{Z}^2(L)\times \mathcal{C}^1(L).\]
As a consequence,
\[\dim_{\mathbb{F}_2}\left(\mathcal{C}^2_\delta(L)\right)=\dim_{\mathbb{F}_2}\left(\mathcal{Z}^2(L)\right)+n.\]
\end{theorem}

\begin{proof} Let us consider the maps
\begin{equation}\label{eq_mu}
\begin{array}{cccc}
\mu:& \mathcal{C}^1(L)& \to & \mathcal{C}^2(L)\\
& f &\mapsto & \begin{array}{cccc}
 \mu_f: & L\times L & \to & \mathbb{Z}_2\\
& (a,b) & \mapsto & \mu_f(a,b):=f(ab)
\end{array}
\end{array}
\end{equation}
\begin{equation}\label{eq_nu}\begin{array}{cccc}
\nu:& \mathcal{C}^1(L)& \to & \mathcal{C}^2(L)\\
& f &\mapsto & \begin{array}{cccc}
 \nu_f: & L\times L & \to & \mathbb{Z}_2\\
& (a,b) & \mapsto & \nu_f(a,b):=f(a)+f(b)
\end{array}
\end{array}
\end{equation}
and
\begin{equation}\label{eq_theta}\begin{array}{cccc}
\Theta:& \mathcal{Z}^2(L)\times  \mathcal{C}^1(L)& \to & \mathcal{C}_\delta^2(L)\\
& (\psi,f) &\mapsto & \Theta_{\psi,f}:=(\psi+\delta_f,\,\psi+\nu_f)
\end{array}
\end{equation}
We claim that the map $\Theta$ is an isomorphism. First, we prove that $\Theta$ is well-defined. Since $\psi\in \mathcal{Z}^2(L)$, we have
\[\delta_{\psi+\nu_f}(a,b,c)=\delta_{\nu_f}(a,b,c)=f(ab)+f(bc)=\mu_f(a,b)+\mu_f(b,c)\]
for all $a,b,c\in L$. Furthermore, since $\delta_f=\mu_f+\nu_f$, we have
\begin{equation}\label{eq_mu_f}
\mu_f=(\psi+\delta_f) + (\psi+\nu_f).
\end{equation}
Hence, Lemma \ref{lem:equivalence} implies $\Theta_{\psi,f}\in \mathcal{C}_\delta^2(L)$.

Second, we prove that $\Theta$ is injective. It is easily verified that $\Theta$ is linear. Thus, it is enough to prove that $\mathrm{ker}(\Theta)=\left\{({\bf 0},\,{\bf 0})\right\}$. So, let $(\psi,f)\in \mathcal{Z}^2(L)\times  \mathcal{C}^1(L)$ be such that $\Theta_{\psi,f}=({\bf 0},\,{\bf 0})$. Then, from \eqref{eq_mu_f} we have $\mu_f={\bf 0}$. That is, $f(ab)=0$ for all $a,b\in L$. Since $L$ is a loop, this is equivalent to saying that $f(c)=0$ for all $c\in L$. That is, $f={\bf 0}$. But then $\delta_f={\bf 0}$, and hence $\psi={\bf 0}$.

Finally, we prove that $\Theta$ is onto. To do it, let $(\Psi,\Phi)\in \mathcal{C}_\delta^2(L)$. From the first statement of Lemma \ref{lemma:compatible}, there exists $r\in \mathbb{Z}_2$ such that $\Psi(a,e)=\Psi(e,b)=r$ for all $a,b\in L$. Let $f\in\mathcal{C}^1(L)$ be defined so that
\[f(a):=\Phi(e,a)+r\]
for all $a\in L$. From Lemma \ref{lem:equivalence}, we have
\begin{equation}\label{eq:Phi}
\left(\Psi+\Phi\right)(a,b)+\left(\Psi+\Phi\right)(b,c)=\delta_{\Phi}(a,b,c)
\end{equation}
for all $a,b,c\in L$. If we take $a=e$, then
\begin{align*}
\left(\Psi+\Phi\right)(e,b)+\left(\Psi+\Phi\right)(b,c)=\delta_{\Phi}(e,b,c)\Leftrightarrow\\
(\Psi(e,b)+\Phi(e,b))+\left(\Psi+\Phi\right)(b,c)=\Phi(e,b)+\Phi(b,c)+\Phi(e,bc)+\Phi(b,c)\Leftrightarrow\\
r+\Phi(e,b)+\left(\Psi+\Phi\right)(b,c)=\Phi(e,b)+\Phi(e,bc) \Leftrightarrow\\
\left(\Psi+\Phi\right)(b,c)=\Phi(e,bc)+r\Leftrightarrow\\
\left(\Psi+\Phi\right)(b,c)=f(bc)=\mu_f(b,c).
\end{align*}
That is,
\begin{equation}\label{eq:mu_f}
\mu_f=\Psi+\Phi.
\end{equation}
Then,
\[\delta_{\nu_f}(a,b,c)=f(ab)+f(bc)=\mu_f(a,b)+\mu_f(b,c)=\left(\Psi+\Phi\right)(a,b)+\left(\Psi+\Phi\right)(b,c)\]
for all $a,b,c\in L$. Thus, we have from \eqref{eq:Phi} that $\delta_{\nu_f}=\delta_\Phi$. As a consequence, $\Phi+\nu_f\in \mathcal{Z}^2(L)$. Moreover,
\[\Theta\left(\Phi+\nu_f,f\right)=\left(\Phi+\nu_f+\delta_f,\, \Phi+\nu_f +\nu_f\right)=\left(\Phi+\mu_f,\, \Phi\right)=\left(\Psi,\Phi\right).\]
The final consequence follows straightforwardly.\hfill $\Box$
\end{proof}

\begin{corollary}\label{corollary:symmetric}  There is a vector space isomorphism
\[\mathcal{SC}^2_\delta(L)\cong \mathcal{Z}^2(L)\times \mathbb{Z}_2.\]
As a consequence, \[\dim_{\mathbb{F}_2}\left(\mathcal{SC}^2_\delta(L)\right)=\dim_{\mathbb{F}_2}\left(\mathcal{Z}^2(L)\right)+1.\]
\end{corollary}

\begin{proof} Let $(\Psi,\Phi)\in \mathcal{SC}^2_\delta(L)$. The third statement of Lemma \ref{lemma:compatible} implies that $(\Psi,\Phi)$ is symmetric if and only if $\Psi+\Phi$ is constant. From the proof of Theorem \ref{theorem:compatible_isomorphism} and \eqref{eq:mu_f}, this is uniquely related to a pair $(\psi,f)\in \mathcal{Z}^2(L)\times \mathcal{C}^1(L)$ such that $\Theta_{\psi,f}=(\Psi,\Phi)$, with $\mu_f=\Psi+\Phi$. This implies that $f$ must be constant, and the result holds. \hfill $\Box$
\end{proof}

We finish this section by showing how the algebraic symmetries of the loop $L$, which are described by its automorphism group $\mathrm{Aut}(L)$, preserve the group $\mathcal{C}^2_\delta(L)$. Recall here that an {\em automorphism} of $L$ is a bijection $\pi:L\to L$ such that $\pi(a)\pi(b)=\pi(ab)$ for all $a,b\in L$. The group $\mathrm{Aut}(L)$ acts on the group of $k$-cochains over $L$ as $\pi\cdot f := f^\pi$ for all $\pi\in\mathrm{Aut}(L)$ and $f\in \mathcal{C}^k(L,\mathbb{Z}_2)$, where $f^\pi\in \mathcal{C}^k(L,\mathbb{Z}_2)$ is described so that
    \begin{equation}\label{eq_action}
    f^\pi(a_1,\ldots,a_k) := f(\pi(a_1),\ldots,\pi(a_k))
    \end{equation}
    for all $a_1,\ldots,a_k\in L$.

\begin{proposition}\label{prop:theta_equivariance} It is verified that
\[ \Theta(\psi^\pi, f^\pi)= (\Theta(\psi, f))^\pi\]
for all $(\psi, f) \in \mathcal{Z}^2(L) \times \mathcal{C}^1(L)$ and $\pi\in \mathrm{Aut}(L)$.
\end{proposition}

\begin{proof} Let $(\psi, f) \in \mathcal{Z}^2(L) \times \mathcal{C}^1(L)$ and $\pi\in \mathrm{Aut}(L)$. Then,
\[ \nu_{f^\pi}(a, b) = f^\pi(a) + f^\pi(b) = f(\pi(a)) + f(\pi(b)) = \nu_f(\pi(a), \pi(b)) = (\nu_f)^\pi(a, b), \]
and
{\small\[ \delta_{f^\pi}(a, b) = f^\pi(ab) + f^\pi(a) + f^\pi(b) = f(\pi(a)\pi(b)) + f(\pi(a)) + f(\pi(b)) = \delta_f(\pi(a), \pi(b)) = (\delta_f)^\pi(a, b)\]}
for all $a,b\in L$. Then,
\[\Theta(\psi^\pi, f^\pi) = \left( \psi^\pi + \delta_{f^\pi}, \, \psi^\pi + \nu_{f^\pi} \right) = \left( (\psi + \delta_f)^\pi, \, (\psi + \nu_f)^\pi \right) = \left( \Theta(\psi, f) \right)^\pi.\]
\hfill $\Box$
\end{proof}

\section{Proper multidimensional Hadamard matrices from $\delta$-compatible $2$-cochains over loops}\label{sec:2_loop_cocycles}

In this section, we introduce a construction of proper multidimensional Hadamard matrices from $\delta$-compatible pairs of $2$-cochains over loops, which generalizes the classical construction of multidimensional Hadamard matrices from cocycles described in \eqref{eq:2tod} to the non-associative framework of loops. For each pair $(\psi,\phi)\in\mathcal{C}^2(L)\times \mathcal{C}^2(L)$, and each positive integer $d\geq 2$, we define the $d$-dimensional array $H_{\psi,\phi,d}$ so that, for each $a_1,\ldots,a_d\in L$,
\begin{equation}\label{eq:habc}
H_{\psi,\phi,d}[a_1,\ldots,a_d]:=\prod_{i=2}^{d-1} M_\psi\left[p_{i-1},a_i\right]\cdot M_{\phi}\left[p_{d-1},a_d\right]
\end{equation}
where $p_1:=a_1$ and $p_i:=p_{i-1}a_i$, whenever $2\leq i<d$. This matrix coincides with that described in \eqref{eq:2tod} whenever $\psi=\phi$.

\begin{lemma}\label{lemma_0a} If $(\psi,\phi)\in\mathcal{C}^2_\delta(L)$, then for each positive integer $k<d$,
{\small \begin{align*}
\phantom{=}\, & H_{\psi,\phi,d}[a_1,\ldots,a_{k-1},a,a_{k+1},\ldots,a_d]=\\
=\,& \begin{cases}
    M_\phi\left[a,q_2\right]\cdot \left(\prod_{i=2}^{d-1} M_\psi\left[a_i,q_{i+1}\right]\right), & \text{ if } k=1,\\
\left(\prod_{i=2}^{k-1} M_\psi\left[p_{i-1},a_i\right]\right) \cdot M_\psi\left[p_{k-1},a\right]\cdot M_\phi\left[p_{k-1}a,q_{k+1}\right]\cdot \left(\prod_{i=k+1}^{d-1} M_\psi\left[a_i,q_{i+1}\right]\right), & \text{ if } k>1.
\end{cases}
\end{align*}}
where $q_d:=a_d$ and $q_i:=a_iq_{i+1}$ whenever $k+1\leq i<d$.
\end{lemma}

\begin{proof} The result holds by repeatedly applying \eqref{eq:psi_phi} in \eqref{eq:habc} starting with the rightmost two factors and proceeding successively to the position $k$.\hfill $\Box$
\end{proof}

The following results generalize the construction of $d$-dimensional Hadamard matrices described in \eqref{eq:2tod}. They show that associativity is not required to construct proper $d$-dimensional Hadamard matrices from $\delta$-compatible pairs of $2$-cochains over loops. A preliminary lemma is required to this end.

\begin{lemma}\label{lemma:permutation} For any finite sequence of fixed elements $a_1,\ldots,a_r\in L$, where $r$ is a positive integer, and any choice of side (left/right) in each step, the mapping obtained by composing the corresponding translations is a permutation of $L$.
\end{lemma}

\begin{proof} The result follows readily from the fact that the binary operation of any finite loop has left- and right-divisions, so left and right translations are permutations in any loop. \hfill $\Box$
\end{proof}

\begin{theorem}\label{theorem_0a} If $(\psi,\phi)\in\mathcal{C}^2_\delta(L)$, then every two-dimensional section of $H_{\psi,\phi,d}$ is Hadamard equivalent to $M_{\phi}$.
\end{theorem}

\begin{proof} Since $H_{\psi,\phi,2}=M_{\phi}$, the result holds trivially for $d=2$. So, we assume from now on that $d>2$. For each pair of distinct positive integers $k,l\leq d$, with $k<l$, let $a_1,\ldots,a_{k-1},a_{k+1},\ldots,$ $a_{l-1},a_{l+1},\ldots,a_d\in L$. We claim that the two-dimensional matrix $H$, which is described so that
\[H[a_k,a_l]=H_{\psi,\phi}\left[a_1,\ldots,a_d\right]\]
for all $a_k,a_l\in L$, is Hadamard equivalent to $M_{\phi}$. To prove it, we distinguish a study of cases depending on the positions of $k$ and $l$.

First, we assume $l=d$. From \eqref{eq:habc},
\[H[a_k,a_d]= \epsilon(a_k) \cdot M_{\phi}\left[p_{d-1},a_d\right]\]
for all $a_k,a_d\in L$, where
\[\epsilon(a_k):=\left(\prod_{i=2}^{d-1} M_\psi\left[p_{i-1},a_i\right]\right)\in \{\pm 1\}\]
is a factor that depends only on $a_k$. In addition, the map $a_k\mapsto p_{d-1}$ is a permutation of $L$ from Lemma \ref{lemma:permutation}. Hence, $H$ is Hadamard equivalent to $M_{\phi}$.

Second, we assume $k=1$ and $l<d$. From Lemma \ref{lemma_0a},
\[H[a_1,a_l]= M_{\phi}\left[a_1,q_2\right]\cdot \eta(a_l)\]
for all $a_1,a_l\in L$, where \begin{equation}\label{eq:g}\eta(a_l):=\left(\prod_{i=k+1}^{d-1} M_\psi\left[a_i,q_{i+1}\right]\right)\in \{\pm 1\}
\end{equation}
is a factor that depends only on $a_l$. In addition, the map $a_l\mapsto q_2$ is a permutation of $L$ from Lemma \ref{lemma:permutation}. Hence, $H$ is Hadamard equivalent to $M_{\phi}$.

Finally, we assume $k>1$ and $l<d$. From Lemma \ref{lemma_0a},
\[H[a_k,a_l]= \epsilon(a_k) \cdot M_{\phi}\left[p_{k-1}a_k,q_{k+1}\right]\cdot \eta(a_l)\]
for all $a_k,a_l\in L$, where
\[\epsilon(a_k):=\left(\prod_{i=2}^{k-1} M_\psi\left[p_{i-1},a_i\right]\right)\in \{\pm 1\}\]
and $\eta(a_l)$, described as in \eqref{eq:g}, are two factors that depend, respectively, only on $a_k$ and $a_l$. In addition, both maps $a_k\mapsto p_{k-1}a_k$ and $a_l\mapsto q_{k+1}$ are permutations of $L$ from Lemma \ref{lemma:permutation}. Hence, $H$ is Hadamard equivalent to $M_{\phi}$. \hfill $\Box$
\end{proof}

\begin{corollary}\label{corollary_0} If the matrix $M_\phi$ is Hadamard, then $H_{\psi,\phi,d}$ is proper Hadamard.
\end{corollary}

\begin{proof} The result follows readily from Theorem \ref{theorem_0a}. \hfill $\Box$
\end{proof}

\vspace{0.2cm}

Unlike the multidimensional Hadamard matrices based on $2$-cocycles described by \citealp{Horadam1998}, the $2$-cochain $\phi$ need not be a $2$-cocycle. The following example illustrates this fact.

\begin{example}\label{example_1} It is readily verified that the matrix $M_\phi$ in Example \ref{example_0} is Hadamard. However, $\phi\not\in\mathcal{Z}^2(L)$. Thus, for instance,
\[\phi(1,1)+\phi(1\cdot 1,2)=\phi(1,1)+\phi(1,2)=0+1=1\neq 0=\phi(1,2)+\phi(1,2)=\phi(1,1\cdot 2)+\phi(1,2).\]
Corollary \ref{corollary_0} ensures that the $d$-dimensional matrix $H_{\psi,\phi,d}$, with any $d\geq 2$, is proper Hadamard. Figure \ref{fig:example_0} describes the case $d=3$, where white cells represent the symbol $1$, and black ones the symbol $-1$.
\end{example}

\begin{figure}[htbp]
	\centering
\resizebox{0.3\textwidth}{!}{
\begin{tikzpicture}[
    x={(0.65cm,-0.22cm)},
    y={(0.42cm,0.28cm)},
    z={(0cm,1.2cm)},
    scale=0.55
]

  \drawtensorlayer{1}{
     1,-1,-1, 1,-1,-1, 1, 1,
    -1, 1, 1,-1,-1,-1, 1, 1,
    -1, 1,-1, 1, 1, 1, 1, 1,
     1,-1, 1,-1, 1, 1, 1, 1,
    -1,-1, 1, 1,-1, 1, 1,-1,
    -1,-1, 1, 1, 1,-1,-1, 1,
     1, 1, 1, 1,-1, 1,-1, 1,
     1, 1, 1, 1, 1,-1, 1,-1
  }

  \drawtensorlayer{2}{
    -1, 1, 1,-1,-1,-1, 1, 1,
     1,-1,-1, 1,-1,-1, 1, 1,
     1,-1, 1,-1, 1, 1, 1, 1,
    -1, 1,-1, 1, 1, 1, 1, 1,
    -1,-1, 1, 1, 1,-1,-1, 1,
    -1,-1, 1, 1,-1, 1, 1,-1,
     1, 1, 1, 1, 1,-1, 1,-1,
     1, 1, 1, 1,-1, 1,-1, 1
  }

  \drawtensorlayer{3}{
    -1, 1,-1, 1, 1, 1, 1, 1,
     1,-1, 1,-1, 1, 1, 1, 1,
    -1, 1, 1,-1, 1, 1,-1,-1,
     1,-1,-1, 1, 1, 1,-1,-1,
     1, 1, 1, 1,-1, 1,-1, 1,
     1, 1, 1, 1, 1,-1, 1,-1,
     1, 1,-1,-1, 1,-1,-1, 1,
     1, 1,-1,-1,-1, 1, 1,-1
  }

  \drawtensorlayer{4}{
     1,-1, 1,-1, 1, 1, 1, 1,
    -1, 1,-1, 1, 1, 1, 1, 1,
     1,-1,-1, 1, 1, 1,-1,-1,
    -1, 1, 1,-1, 1, 1,-1,-1,
     1, 1, 1, 1, 1,-1, 1,-1,
     1, 1, 1, 1,-1, 1,-1, 1,
     1, 1,-1,-1,-1, 1, 1,-1,
     1, 1,-1,-1, 1,-1,-1, 1
  }

  \drawtensorlayer{5}{
    -1,-1, 1, 1,-1, 1,-1, 1,
    -1,-1, 1, 1, 1,-1, 1,-1,
     1, 1, 1, 1,-1, 1, 1,-1,
     1, 1, 1, 1, 1,-1,-1, 1,
    -1, 1, 1,-1, 1, 1, 1, 1,
     1,-1,-1, 1, 1, 1, 1, 1,
    -1, 1,-1, 1, 1, 1,-1,-1,
     1,-1, 1,-1, 1, 1,-1,-1
  }

  \drawtensorlayer{6}{
    -1,-1, 1, 1, 1,-1, 1,-1,
    -1,-1, 1, 1,-1, 1,-1, 1,
     1, 1, 1, 1, 1,-1,-1, 1,
     1, 1, 1, 1,-1, 1, 1,-1,
     1,-1,-1, 1, 1, 1, 1, 1,
    -1, 1, 1,-1, 1, 1, 1, 1,
     1,-1, 1,-1, 1, 1,-1,-1,
    -1, 1,-1, 1, 1, 1,-1,-1
  }

  \drawtensorlayer{7}{
     1, 1, 1, 1, 1,-1,-1, 1,
     1, 1, 1, 1,-1, 1, 1,-1,
     1, 1,-1,-1,-1, 1,-1, 1,
     1, 1,-1,-1, 1,-1, 1,-1,
     1,-1, 1,-1, 1, 1,-1,-1,
    -1, 1,-1, 1, 1, 1,-1,-1,
    -1, 1, 1,-1,-1,-1,-1,-1,
     1,-1,-1, 1,-1,-1,-1,-1
  }

  \drawtensorlayer{8}{
     1, 1, 1, 1,-1, 1, 1,-1,
     1, 1, 1, 1, 1,-1,-1, 1,
     1, 1,-1,-1, 1,-1, 1,-1,
     1, 1,-1,-1,-1, 1,-1, 1,
    -1, 1,-1, 1, 1, 1,-1,-1,
     1,-1, 1,-1, 1, 1,-1,-1,
     1,-1,-1, 1,-1,-1,-1,-1,
    -1, 1, 1,-1,-1,-1,-1,-1
  }

  \coordinate (O) at (-1.5, -1.5, -0.5);

  \draw[->, >=stealth, line width=1pt, black] (O) -- (9.5, -1.5, -0.5) node[below right, font=\small\bfseries] {Columns};

  \draw[->, >=stealth, line width=1pt, black] (O) -- (-1.5, 9.5, -0.5) node[above left, font=\small\bfseries] {Rows};

  \draw[->, >=stealth, line width=1pt, black] (O) -- (-1.5, -1.5, 7*3.4 + 2) node[above, font=\small\bfseries] {Slices};
  \foreach \z in {1,...,8} {
    \pgfmathsetmacro{\zpos}{(\z - 1) * 3.4}
    \draw[thick] (-1.5, -1.5, \zpos) -- (-1.8, -1.5, \zpos) node[left, font=\scriptsize] {\z};
  }

\end{tikzpicture}}\\
	\caption{Representation of the $3$-dimensional proper Hadamard matrix $H_{\psi,\phi,3}$ in Example \ref{example_0}.}
	\label{fig:example_0}
\end{figure}
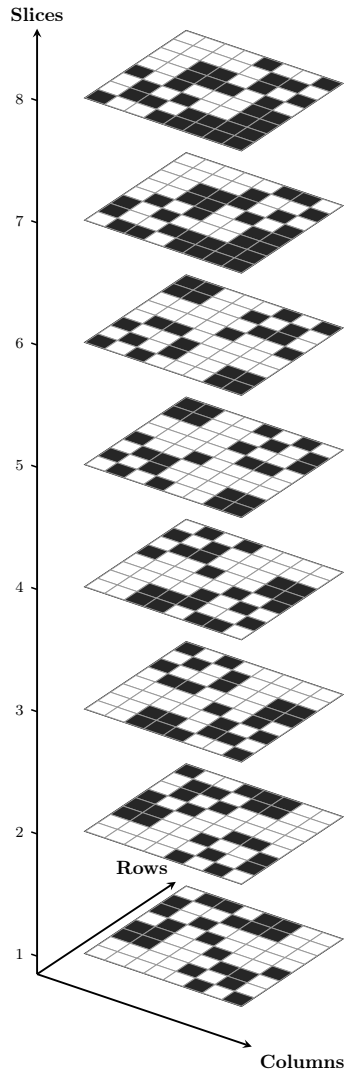

It is worth highlighting the relationship between the $\delta$-compatible pairs of $2$-cochains $(\psi,\phi_{\psi,f})$ described in Lemma \ref{lemma:compatible_a} and the classical associative setting described in \citealp[Proposition~4.1]{Horadam1998}. In the group-cocyclic framework, any arbitrary $1$-cochain $f$ yields a valid candidate matrix $M_{\phi_{\psi,f}}$, because $\delta_\psi=\delta^2(f)=0$ holds trivially for any choice of $f$. In contrast, for non-associative loops, $\delta_\psi$ and $\delta^2(f)$ are non-zero in general, so the loop-cocycle condition forces $f$ to precisely mirror and cancel the associator obstruction of $\psi$, highlighting the non-trivial nature of the loop generalization. The next result shows that the Hadamard condition in $M_{\phi_{\psi,f}}$ in Corollary \ref{corollary_0} is indeed derived from an ordinary $2$-cocycle.

\begin{lemma}\label{lemma:M_psi} The matrix $M_{\phi_{\psi,f}}$ is Hadamard equivalent to the cocyclic matrix $M_{\psi+\delta_f}$. As a consequence, $M_{\phi_{\psi,f}}$ is Hadamard if and only if $M_{\psi +\delta_f}$ is Hadamard.
\end{lemma}

\begin{proof} Since $\psi\in \mathcal{Z}_\mathcal{L}^2(L)$, we have from \eqref{eq:psi_delta} that $\psi+\delta_f\in \mathcal{Z}^2(L)$, which implies that the matrix $M_{\psi+\delta_f}$ is $2$-cocyclic over $L$. Furthermore, since $\delta_f(a,b) = f(a) + f(b) + f(ab)$ for all $a,b\in L$, we have
\begin{align*}
    M_{\psi+\delta_f}[a,b] &= (-1)^{\psi(a,b) + \delta_f(a,b)} \\
    &= (-1)^{\psi(a,b) + f(a) + f(b) + f(a \cdot b)} \\
    &= (-1)^{f(a)} \cdot (-1)^{\psi(a,b) + f(a \cdot b)} \cdot (-1)^{f(b)} \\
    &= (-1)^{f(a)} \cdot M_{\phi_{\psi,f}}[a,b] \cdot (-1)^{f(b)}.
\end{align*}
Equivalently,
\[M_{\psi+\delta_f} = D_L \cdot M_{\phi_{\psi,f}} \cdot D_R,\]
where $D_L = \operatorname{diag}\left((-1)^{f(a)}\right)_{a \in L}$ and $D_R = \operatorname{diag}\left((-1)^{f(b)}\right)_{b \in L}$ are diagonal signature matrices with entries in $\{-1, 1\}$. Since $M_{\psi+\delta_f}$ is obtained from $M_{\phi_{\psi,f}}$ solely by row and column negations, both matrices are Hadamard equivalent. The consequence follows because Hadamard equivalence preserves the orthogonality of rows and columns. \hfill $\Box$
\end{proof}

\vspace{0.2cm}

Although the Hadamard property of the matrix $H_{\psi,\phi_{\psi,f},d}$ is ultimately determined by the ordinary cocycle $\psi+\delta_f\in \mathcal{Z}^2(L)$, the multidimensional construction retains the non-associative information encoded by the loop-cocycle $\psi\in\mathcal{Z}_\mathcal{L}^2(L)$. Thus, loop-cocycles are not needed to enlarge the class of fundamental Hadamard matrices, but they provide the natural algebraic framework for their multidimensional realization over non-associative loops. In this regard, Lemma \ref{lemma:M_psi} is particularly relevant from a computational point of view. On the one hand, we have from \eqref{eq:decomposition} that
\[\mathrm{dim}_{\mathbb{F}_2}\left(\mathcal{Z}^k(L)\right)\leq \mathrm{dim}_{\mathbb{F}_2}\left(\mathcal{Z}_\mathcal{L}^k(L)\right)\]
so all potential Hadamard matrices (up to equivalence) are covered by testing a smaller space of pure $2$-cocycles. On the other hand, the Hadamard test is known to be valid in the loop framework for cocycles (see \citealp{Alvarez2019}), but not for loop-cocycles (see \citealp{Falcon26}). In this way, the computational cost of testing whether the underlying matrix is Hadamard is $O(n^2)$, whereas a direct verification of all required orthogonality conditions on the $d$-dimensional array involves $O\left(n^{2(d-1)}\right)$ scalar products.

\section{Computational analysis}\label{sec:computational}

This section focuses on a first computational approach to identify and classify the multidimensional Hadamard matrices described by \eqref{eq:habc}, which arise from $\delta$-compatible pairs of $2$-cochains over a finite loop $L$ of order $n$. A key computational consideration here is the size of the vector space $\mathcal{C}_\delta^2(L)$. According to Theorem \ref{theorem:compatible_isomorphism}, an exhaustive search across the entire domain would require evaluating $2^{\dim_{\mathbb{F}_2} \mathcal{Z}^2(L) + n}$ candidates, which rapidly becomes computationally intractable as $n$ increases. To reduce this exponential complexity, we exploit the isomorphism $\Theta$ described in \eqref{eq_theta}. Every pair $(\Psi,\Phi)\in\mathcal C_\delta^2(L)$ has a unique representation
\[
(\Psi,\Phi)=(\psi+\delta_f,\psi+\nu_f),\]
with $(\psi,f)\in\mathcal Z^2(L)\times\mathcal C^1(L)$. For the multidimensional construction in
\eqref{eq:habc}, the corresponding array satisfies
\[
H_{\psi+\delta_f,\psi+\nu_f,d}[a_1,\ldots,a_d]
=
H_{\psi,\psi,d}[a_1,\ldots,a_d]
\prod_{j=1}^d(-1)^{f(a_j)}.
\]
Equivalently, in matrix form,
\[
H_{\psi+\delta_f,\psi+\nu_f,d}
=
(D_f\otimes\cdots\otimes D_f)H_{\psi,\psi,d}.
\]
where
\[
D_f=\operatorname{diag}\bigl((-1)^{f(a)}:a\in L\bigr).
\]
Thus, the $1$-cochain $f$ only induces sign changes of parallel hyperplanes along the coordinate directions. Consequently, both arrays $H_{\psi+\delta_f,\psi+\nu_f,d}$ and $H_{\psi,\psi,d}$ are equivalent under the multidimensional Hadamard equivalence, which we recall below.

\begin{definition}[\citealp{deLauney2008}] Two $d$-dimensional Hadamard matrices $H$ and $H'$ of order $n$ are {\em Hadamard equivalent} if there exist, for each $i\in\{1,\ldots,d\}$, a permutation
$\pi_i$ of $L$ and a function
$\epsilon_i:L\longrightarrow\{\pm1\}$ such that
\[
H'[a_1,\ldots,a_d]
=
\left(\prod_{i=1}^d\epsilon_i(a_i)\right)
H[\pi_1(a_1),\ldots,\pi_d(a_d)]\]
for all $a_1,\ldots,a_d\in L$. Equivalently, if $P_i$ is a permutation matrix and $D_i$ is a
diagonal matrix with diagonal entries in $\{\pm1\}$, then
\[
H'=
(D_1P_1)\otimes\cdots\otimes(D_dP_d)\,H.
\]
\end{definition}

Consequently, for the purpose of classifying the resulting $d$-dimensional Hadamard matrices up to multidimensional Hadamard equivalence, it is sufficient to consider the pairs $(\psi,\psi)$ with $\psi\in\mathcal Z^2(L)$. The $1$-cochains $f \in \mathcal{C}^1(L)$ simply populate existing equivalence classes with isomorphic matrix instances. Fixing $f = 0$ reduces the computational search space from $|\mathcal{Z}^2(L)| \cdot 2^n$ to simply $|\mathcal{Z}^2(L)|$. This eliminates a redundant factor of $2^n$ from the search space while guaranteeing that no distinct $d$-dimensional Hadamard equivalence class is overlooked.

We have implemented a systematic search framework in the {\sc GAP} system {\em (Groups, Algorithms, Programming)} (see \citealp{GAP}) for the computation of $3$-dimensional Hadamard equivalence classes based on $\delta$-compatible pairs of $2$-cochains over all groups of order $4$ and $8$, and the non-associative loop described in Example \ref{example_0}. More precisely, for a given finite loop $L$ of order $n$, the implementation proceeds in three main stages. First, we determine the vector space $\mathcal{Z}^2(L)$ by solving the linear system of equations \eqref{eq:system_delta} once we impose $x_{a,b}=y_{a,b}$ for all $a,b\in L$. Second, we classify, up to Hadamard equivalence, all those solutions satisfying the Hadamard condition. Third, the corresponding $3$-dimensional arrays are constructed and classified up to Hadamard equivalence.

For the computational classification of Hadamard matrices, we have implemented in the GRAPE package \textsc{GRAPE} (see \citealp{GRAPE}) in {\sc GAP} a natural generalization of the classical construction described by \citealp{McKay1979}, which reduces the Hadamard equivalence to a coloured graph isomorphism. (We refer to \citealp{Ostergaard2026} for a recent implementation to classify complex Hadamard matrices.) More precisely, for a $d$-dimensional Hadamard array, we associate two signed vertices with each index in each coordinate direction and encode each array entry by the $2^{d-1}$ sign patterns whose product agrees with the corresponding entry. For $d=2$, this reduces to the classical McKay graph representation, while the higher-dimensional construction provides a natural graph encoding of the multidimensional equivalence considered in this paper. The graph representation is used only as a computational device for testing equivalence. The mathematical results concerning the construction of the arrays and the reduction of the search space are established in the preceding sections. In any case, our implementation keeps independent verification routines for validation purposes. Algorithm \ref{alg:multidimensional-hadamard} describes the pseudocode of our procedure.

\begin{algorithm}
\caption{Enumeration and classification of $d$-dimensional Hadamard matrices over a finite loop}
\label{alg:multidimensional-hadamard}
\begin{algorithmic}[1]
\Require A finite loop $L$ of order $n$, and a dimension $d\geq 2$.
\Ensure The set $\mathcal{R}$ of representatives of $d$-dimensional Hadamard equivalence classes arising from $\mathcal{C}_\delta^2(L)$.

\State Construct $\mathcal{Z}^2(L)$ from the linear system over $\mathbb{F}_2$.
\State Initialize $\mathcal{R}=\emptyset$.

\ForAll{$\psi\in \mathcal{Z}^2(L)$}
    \State Construct the cocyclic matrix $M_\psi$.
    \If{$M_\psi$ is Hadamard}
        \State Construct the $d$-dimensional array $H_{\psi,\psi,d}$ and its associated coloured graph.
        \If{$H_{\psi,\psi,d}$ is not equivalent to any representative in $\mathcal{R}$}
            \State Add $H_{\psi,\psi,d}$ to $\mathcal{R}$.
        \EndIf
    \EndIf
\EndFor
\State \Return $\mathcal{R}$.
\end{algorithmic}
\end{algorithm}

To ensure full reproducibility of the computational results, the complete GAP/GRAPE implementation used in this work is provided as supplementary material accompanying the article. The procedure is applied independently to each loop under consideration. The resulting representatives are then compared across different loops, again using graph isomorphisms. Table~\ref{tab:gap} summarizes the computational results obtained with our {\sc GAP}/{\sc GRAPE} implementation of Algorithm~\ref{alg:multidimensional-hadamard} (with $d=3$) for all groups of order $4$ and $8$, together with the non-associative loop of order $8$ described in Example \ref{example_0}. The column $\dim_{\mathbb F_2}\mathcal Z^2(L)$ gives the dimension of the space of $2$-cocycles over $L$. The column {\em ``2D Hadamard matrices''} indicates the total number of cocyclic Hadamard matrices obtained from Algorithm~\ref{alg:multidimensional-hadamard}. The columns {\em ``2D classes''} and {\em ``3D classes''} give the number of equivalence classes obtained through the graph-isomorphism classification. The column {\em ``Time''} reports the total CPU execution time in seconds. All experiments were performed on a workstation equipped with a {\em $13^{\mathrm{th}}$ Gen Intel Core i9-13900H CPU @ 2.60GHz with 32 GB of RAM}. Note in particular that no non-trivial cocyclic Hadamard matrix was found over the cyclic group $\mathbb{Z}_8$ and the quaternion group $\mathcal Q_8$. Among the remaining examples, the elementary abelian group $\mathbb Z_2\times \mathbb Z_2\times \mathbb Z_2$ yields the richest family of multidimensional Hadamard matrices, producing 336 cocyclic Hadamard matrices distributed into four inequivalent three-dimensional classes.

The computational data show that the proposed multidimensional construction yields a strict refinement of the classical cocyclic Hadamard classification. While all cocyclic Hadamard matrices of orders $4$ and $8$ (among the loops that do produce at least one such matrix) collapse into a single global equivalence class in dimension two, the associated three-dimensional arrays split into four and eleven non-equivalent classes, respectively. Moreover, no two three-dimensional Hadamard matrices arising from different loops were found to be Hadamard equivalent. Consequently, the three-dimensional construction distinguishes all loop sources occurring in our experiments, including the non-associative loop of Example~\ref{example_0}. In particular, the latter produces a three-dimensional Hadamard equivalence class that does not occur for any group of order $8$.

\renewcommand{\tabcolsep}{4pt}
\begin{table}[htbp]
\centering
\caption{Computational results.} \label{tab:gap}
{\scriptsize 
\begin{tabular}{clcccccr} \hline
Order & Loop $L$ & $\dim_{\mathbb F_2}\mathcal Z^2(L)$ & 2D Hadamard matrices & 2D classes  & 3D classes & Global 3D classes & Time (s) \\ \hline
4 & $\mathbb Z_4$ & 4 & 4 & 1 & 1 & $G_{4.1}$ & 0.281 \\
& $\mathbb{Z}_2\times\mathbb{Z}_2$ & 5 & 12 & 1 & 3 & $G_{4.2}$ - $G_{4.4}$ & 1.094 \\ \hline
8 & $\mathbb Z_8$ & 8 & 0 & 0 & 0 & & 0.000 \\
& $\mathbb Z_4\times \mathbb Z_2$ & 9 & 32 & 1 & 3 & $G_{8.1}$ - $G_{8.3}$& 24.484 \\
& $\mathcal{D}_8$ & 9 & 64 & 1 & 3 & $G_{8.4}$ - $G_{8.6}$& 76.609\\
& $\mathcal{Q}_8$ & 8 & 0 & 0 & 0 & & 0.000\\
& $\mathbb Z_2\times \mathbb Z_2\times \mathbb Z_2$ & 11 & 336 & 1 & 4 & $G_{8.7}$ - $G_{8.10}$&  498.375 \\
& $L$ (Example \ref{example_0}) & 6 & 8 & 1 & 1 & $G_{8.11}$ &  11.312\\
\hline
\end{tabular}}
\end{table}

\section{Conclusion and further work}

In this paper, we have introduced a multidimensional extension of the loop-cocyclic approach to Hadamard matrices. To this end, we have defined the notion of a $\delta$-compatible pair of $2$-cochains over a loop $L$ as a natural extension of the usual $2$-cocycle identity, which allows the use of two distinct $2$-cochains. We have proved that these pairs admit a natural decomposition into a cocyclic component and a free cochain component. This decomposition yields a generalization of the classical group-cocyclic construction of multidimensional Hadamard matrices to arbitrary finite loops while preserving its computational advantages. The obtained decomposition also provides a significant reduction of the search space. More precisely, the determination of multidimensional Hadamard equivalence classes can be restricted to the space of ordinary $2$-cocycles, eliminating a redundant factor of $2^n$. This reduction has enabled the implementation of an efficient computational framework in {\sc GAP}/{\sc GRAPE}, which has been applied to all groups of orders $4$ and $8$, together with a non-associative loop of order $8$. The computational results indicate that the proposed three-dimensional construction captures structural information that is not visible at the level of cocyclic Hadamard matrices. In all cases considered, the two-dimensional classification collapses into a single global equivalence class, whereas the corresponding three-dimensional arrays split into several inequivalent classes. In particular, the non-associative loop considered in this work yields a three-dimensional Hadamard class that does not arise from any of the groups analysed.

Future research will focus on extending the computational classification to larger loop orders and higher dimensions. Another natural direction is the development of a multidimensional analogue of the cocyclic Hadamard test that avoids the explicit construction of multidimensional arrays. The relationship between loop-theoretic properties and the resulting multidimensional Hadamard equivalence classes also deserves further investigation, as the present computational evidence suggests that the three-dimensional construction provides a finer invariant than the associated cocyclic Hadamard matrix.

\section*{Acknowledgements}

The first author has been partially supported by the Research Project {\em ``Modeling small-world, scale-free networks from combinatorial designs based on quasigroup digraphs''} (PPIT-FEDER- SOL2024-31611), co-financed by the EU -- Ministry of Finance and Public Administration -- European Funds -- Andalusian Regional Government -- Ministry of University, Research and Innovation. The second author has partially been supported by the Research Project {\em HADAMARD-CRYPT} (PID2024-160735NB-I00), co-financed by the EU -- Ministry of Finance and Public Administration -- European Funds -- Ministry of University, Research and Innovation.

\bibliographystyle{plainnat}
\bibliography{ref_arxiv}

@misc{alpoge2026,
  author       = {Alp{\"o}ge, L. and Voinov, P. and Reynolds-Haertle, S. and Claude-Anthropic},
  title        = {Explicit construction of {H}adamard matrices for all unknown orders under 2000, including order 668},
  howpublished = {Electronic announcement and source code repository},
  month        = {August},
  year         = {2026},
  note         = {Available online; order 668 verified via FrontierMath benchmarks}
}

@article {Alvarez2015,
    AUTHOR = {\'Alvarez, V. and Armario, J. A. and Frau, M. D. and Real, P.},
     TITLE = {On higher dimensional cocyclic {H}adamard matrices},
   JOURNAL = {Appl. Algebra Engrg. Comm. Comput.},
  FJOURNAL = {Applicable Algebra in Engineering, Communication and
              Computing},
    VOLUME = {26},
      YEAR = {2015},
    NUMBER = {1-2},
     PAGES = {191--206},
       doi={10.1007/s00200-014-0242-3}
}

@article {Alvarez2019,
    AUTHOR = {\'Alvarez, V. and Falc\'on, R. M. and Frau, M. D. and Gudiel,
              F. and G\"uemes, B.},
     TITLE = {Cocyclic {H}adamard matrices over {L}atin rectangles},
   JOURNAL = {European J. Combin.},
  FJOURNAL = {European Journal of Combinatorics},
    VOLUME = {79},
      YEAR = {2019},
     PAGES = {75--96},
       doi={10.1016/j.ejc.2018.12.007}
}

@article{Alvarez2020,
  author    = {{\'A}lvarez, V. and Armario, J. A. and Falc{\'o}n, R. M. and Frau, M. D. and Gudiel, F. and G{\"u}emes, M. B. and Osuna, A.},
  title     = {On Cocyclic {H}adamard Matrices over {G}oethals-{S}eidel Loops},
  journal   = {Mathematics},
  volume    = {8},
  number    = {1},
  pages     = {24},
  year      = {2020},
  doi={10.3390/math8010024}
}

@article {Eilenberg1947,
    AUTHOR = {Eilenberg, S. and MacLane, S.},
     TITLE = {Algebraic cohomology groups and loops},
   JOURNAL = {Duke Math. J.},
  FJOURNAL = {Duke Mathematical Journal},
    VOLUME = {14},
      YEAR = {1947},
     PAGES = {435--463},
       URL = {http://projecteuclid.org/euclid.dmj/1077474141},
}

@misc{epoch2026,
  author       = {{Epoch AI}},
  title        = {FrontierMath Open Problems: Hadamard Matrix of Order 668},
  howpublished = {\url{https://epoch.ai/frontiermath/open-problems/hadamard}},
  year         = {2026},
  note         = {Accessed: September 2026; provisionally marked as solved by AI-human collaboration}
}

@article{Falcon2021,
  author    = {Falc{\'o}n, R. M. and {{\'A}}lvarez, V. and Frau, M. D. and Gudiel, F. and G{\"u}emes, M. B.},
  title     = {Pseudococyclic Partial Hadamard Matrices over Latin Rectangles},
  journal   = {Mathematics},
  volume    = {9},
  number    = {2},
  pages     = {113},
  year      = {2021},
  doi={10.3390/math9020113}
}

@misc{Falcon26,
  author    = {Falc{\'o}n, R. M. and Gonz{\'a}lez-Regadera, M. and Gudiel, F.},
  title     = {Fundamentals of the cocyclic development of {H}adamard matrices over loops},
  year      = {2026},
  note = {Preprint}, 
  doi={10.48550/arXiv.2609.08553}
}

@manual{GAP,
    author = "{The GAP Group}",
    title        = "{GAP--Groups, Algorithms, and Programming,
                    Version 4.16.1}",
    year         = 2026,
    url          = {https://www.gap-system.org}
    }

@misc{ GRAPE,
  author =           {Soicher, L. H.},
  title =            {{GRAPE}, {GRaph Algorithms using PErmutation groups}, {V}ersion 4.9.3},
  month =            {Sep},
  year =             {2025},
  note =             {GAP package},
  howpublished =     {\href                                       {https://gap-packages.github.io/grape}
                      {\texttt{https://gap\texttt{\symbol{45}}packages.github.io/}\discretionary
                      {}{}{}\texttt{grape}}},
  printedkey =       {Soi25}
}

@incollection {Horadam1995,
    AUTHOR = {Horadam, K. J. and de Launey, W.},
     TITLE = {Generation of cocyclic {H}adamard matrices},
 BOOKTITLE = {Computational algebra and number theory ({S}ydney, 1992)},
    SERIES = {Math. Appl.},
    VOLUME = {325},
     PAGES = {279--290},
 PUBLISHER = {Kluwer Acad. Publ., Dordrecht},
      YEAR = {1995}
}

@article {Horadam1998,
    AUTHOR = {Horadam, K. J. and Lin, Cantian},
     TITLE = {Construction of proper higher-dimensional {H}adamard matrices
              from perfect binary arrays},
   JOURNAL = {J. Combin. Math. Combin. Comput.},
  FJOURNAL = {Journal of Combinatorial Mathematics and Combinatorial
              Computing},
    VOLUME = {28},
      YEAR = {1998},
     PAGES = {237--248},
}

@book {Horadam2007,
    AUTHOR = {Horadam, K. J.},
     TITLE = {Hadamard matrices and their applications},
 PUBLISHER = {Princeton University Press, Princeton, NJ},
      YEAR = {2007},
     PAGES = {xiv+263}
}

@article {Johnson1990,
    AUTHOR = {Johnson, K. W. and Leedham-Green, C. R.},
     TITLE = {Loop cohomology},
   JOURNAL = {Czechoslovak Math. J.},
  FJOURNAL = {Czechoslovak Mathematical Journal},
    VOLUME = {40(115)},
      YEAR = {1990},
    NUMBER = {2},
     PAGES = {182--194}
}

@article {Krcadinac2023,
    AUTHOR = {Kr\v{c}adinac, V. and Pav\v{c}evi\'c, M. O. and Tabak,
              K.},
     TITLE = {Three-dimensional {H}adamard matrices of {P}aley type},
   JOURNAL = {Finite Fields Appl.},
  FJOURNAL = {Finite Fields and their Applications},
    VOLUME = {92},
      YEAR = {2023},
     PAGES = {Paper No. 102306, 6},
       doi={10.1016/j.ffa.2023.102306}
}

@article {Krcadinac2025,
    AUTHOR = {Kr\v{c}adinac, V. and Pav\v{c}evi\'c, M. O. and Tabak,
              K.},
     TITLE = {Cubes of symmetric designs},
   JOURNAL = {Ars Math. Contemp.},
  FJOURNAL = {Ars Mathematica Contemporanea},
    VOLUME = {25},
      YEAR = {2025},
    NUMBER = {1},
     PAGES = {Paper No. 10, 16},
       doi={10.26493/1855-3974.3222.e53}
}

@phdthesis{deLauney87,
    author       = {de Launey, W.},
    title        = {\((0,G)\)-designs and applications},
    school       = {University of Sydney},
    year         = {1987}}

@article {deLauney2008,
    AUTHOR = {de Launey, W. and Stafford, R. M.},
     TITLE = {Automorphisms of higher-dimensional {H}adamard matrices},
   JOURNAL = {J. Combin. Des.},
  FJOURNAL = {Journal of Combinatorial Designs},
    VOLUME = {16},
      YEAR = {2008},
    NUMBER = {6},
     PAGES = {507--544},
       doi={10.1002/jcd.20200}
}

@article {McKay1979,
    AUTHOR = {McKay, B. D.},
     TITLE = {Hadamard equivalence via graph isomorphism},
   JOURNAL = {Discrete Math.},
  FJOURNAL = {Discrete Mathematics},
    VOLUME = {27},
      YEAR = {1979},
    NUMBER = {2},
     PAGES = {213--214},
       doi={10.1016/0012-365X(79)90113-4}
}

@article {OC11,
    AUTHOR = {\'O{} Cath\'ain, P. and R\"oder, M.},
     TITLE = {The cocyclic {H}adamard matrices of order less than 40},
   JOURNAL = {Des. Codes Cryptogr.},
  FJOURNAL = {Designs, Codes and Cryptography. An International Journal},
    VOLUME = {58},
      YEAR = {2011},
    NUMBER = {1},
     PAGES = {73--88},
       doi={10.1007/s10623-010-9385-9}
}

@article {Ostergaard2026,
    AUTHOR = {\"Ostergard, P. R. J. and Valtonen, T.},
     TITLE = {Equivalence of complex {H}adamard matrices},
   JOURNAL = {J. Algebraic Combin.},
  FJOURNAL = {Journal of Algebraic Combinatorics. An International Journal},
    VOLUME = {64},
      YEAR = {2026},
    NUMBER = {1},
     PAGES = {Paper No. 20, 18},
       doi={10.1007/s10801-026-01568-x},
}

@article {Shlichta71,
    AUTHOR = {Shlichta, P. J.},
     TITLE = {Three and four-dimensional {H}adamard matrices},
   JOURNAL = {Bull. Am. Phys. Soc.},
    VOLUME = {1},
      YEAR = {1971},
    NUMBER = {16},
     PAGES = {825--826}
}

@article {Shlichta79,
    AUTHOR = {Shlichta, P. J.},
     TITLE = {Higher dimensional {H}adamard matrices},
   JOURNAL = {IEEE Trans. Inform. Theory},
  FJOURNAL = {Institute of Electrical and Electronics Engineers.
              Transactions on Information Theory},
    VOLUME = {25},
      YEAR = {1979},
    NUMBER = {5},
     PAGES = {566--572},
       doi={10.1109/TIT.1979.1056083}
}

@article{Yang86,
    AUTHOR = {Yang, Y. X.},
     TITLE = {The proofs of some conjectures on higher dimensional {H}adamard matrices},
   JOURNAL = {Kexue Tongbao},
  FJOURNAL = {Kexue Tongbao},
    VOLUME = {31},
      YEAR = {1986},
     PAGES = {1662--1667}
}

@book {Yang01,
    AUTHOR = {Yang, Y. X.},
     TITLE = {Theory and applications of higher-dimensional {H}adamard
              matrices},
    SERIES = {Combinatorics and Computer Science},
 PUBLISHER = {Kluwer Academic Publishers Group, Dordrecht; Science Press
              Beijing, Beijing},
      YEAR = {2001},
     PAGES = {xii+319}
}

\end{document}